\documentclass[12pt]{article}

\usepackage{amsmath}
\usepackage{amsfonts}
\usepackage{amssymb}
\usepackage{amsthm}
\usepackage[T1]{fontenc}

\usepackage[left=1in,right=1in,bottom=1in,top=1in]{geometry}
\usepackage{amsmath}
\usepackage{amsfonts}
\usepackage{amssymb}
\usepackage{amsthm}
\usepackage{young}
\usepackage[vcentermath]{youngtab}
\usepackage{tikz}

\usepackage{graphicx}

\newcommand{\equ}{\operatorname{\acute equ}}

\newtheorem{theorem}{Theorem}
\newtheorem{lemma}[theorem]{Lemma}
\newtheorem{example}[theorem]{Example}

\newtheorem{definition}[theorem]{Definition}

\newtheorem{corollary}[theorem]{Corollary}
\newtheorem{conjecture}[theorem]{Conjecture}
\newtheorem{op}[theorem]{Open Question}

\newcommand{\bigzero}{\mbox{\normalfont\Large\bfseries 0}}
\newcommand{\rvline}{\hspace*{-\arraycolsep}\vline\hspace*{-\arraycolsep}}

\def\unprotectedboldentry#1{\textcolor{Red}{\large{#1}}}
\def\boldentry{\protect\unprotectedboldentry}
\usepackage{tikz}
\usetikzlibrary{calc}
\usepackage{ifthen}
\newcommand{\tikztableauinternal}[1]{
    \def\newtableau{#1}
    \coordinate (x) at (-0.5,0.5);
    \coordinate (y) at (-0.5,0.5);
    \foreach \row in \newtableau {
        \coordinate (x) at ($(x)-(0,1)$);
        \coordinate (y) at (x);
        \foreach \entry in \row {
            \ifthenelse{\equal{\entry}{X}}
               {
                \node (y) at ($(y) + (1,0)$) {};
                \fill[color=gray!10] ($(y)-(0.5,0.5)$) rectangle +(1,1);
                \draw[color=gray, dotted] ($(y)-(0.5,0.5)$) rectangle +(1,1);
               }
               {
                \ifthenelse{\equal{\entry}{\boldentry X}}
                   {
                    \node (y) at ($(y) + (1,0)$) {};
                    \fill[color=gray] ($(y)-(0.5,0.5)$) rectangle +(1,1);
                    \draw ($(y)-(0.5,0.5)$) rectangle +(1,1);
                   }
                   {
                    \node (y) at ($(y) + (1,0)$) {\entry};
                    \draw ($(y)-(0.5,0.5)$) rectangle +(1,1);
                   }
               }
            }
        }
}

\newlength\cellsize 
\savebox2{%
\begin{picture}(12,12)
\put(0,0){\line(1,0){12}}
\put(0,0){\line(0,1){12}}
\put(12,0){\line(0,1){12}}
\put(0,12){\line(1,0){12}}
\end{picture}}

\newcommand\cellify[1]{\def\thearg{#1}\def\nothing{}%
\ifx\thearg\nothing
\vrule width0pt height\cellsize depth0pt\else
\hbox to 0pt{\usebox2\hss}\fi%
\vbox to 12\unitlength{
\vss
\hbox to 12\unitlength{\hss$#1$\hss}
\vss}}

\newcommand\tableau[1]{\vtop{\let\\=\cr
\setlength\baselineskip{-12000pt}
\setlength\lineskiplimit{12000pt}
\setlength\lineskip{0pt}
\halign{&\cellify{##}\cr#1\crcr}}}

\title{The Limit Profile of Star Transpositions}

\author{Evita Nestoridi$^{\ast}$
}

\begin{document}

\maketitle

\begin{abstract}
We prove that the limit profile of star transpositions at time $t= n \log n +cn$ is equal to $d_{\textup{T.V.}} (\mbox{Poiss}(1+e^{-c}), \mbox{Poiss}(1))$. We prove this by developing a technique for comparing the limit profile behavior of two reversible Markov chains on the same space, that share the same stationary distribution and eigenbasis. This allows us to compare the limit profile of star transpositions to the limit profile of random transpositions, as studied in \cite{Teyssier}, and prove that they have the same limit profile at the respective cutoff times. 
\end{abstract}


\let\thefootnote\relax
\footnotetext{  $^{\ast}$ \textit{Princeton University, United States. E-Mail}: exn@princeton.edu, \\ 
\phantom{.} \hspace{0.85cm}Supported by NSF Grant DMS-$2052659$.}

\section{Introduction}
Cutoff, a phenomenon according to which a Markov chain converges abruptly to the stationary measure, has been a central question in Markov chain mixing (see \cite{JS} for a nice exposition on the history of cutoff and recent developements). 
Recently, there has been exciting progress towards answering the even sharper question of determining the limit profile of various Markov chains. The limit profile captures the exact shape of the distance of the Markov chain from stationarity. In recent progress, Teyssier \cite{Teyssier} derived an exact formula for the limit profile of random transpositions, improving the seminal result of Diaconis and Shahshahani \cite{DiaSha} and proving a conjecture of Matthews \cite{M}. He extended the Fourier transform arguments of \cite{DiaSha} to use representation theory in order to study limit profiles of conjugacy class walks. In \cite{NT}, Teyssier's representation theory technique is generalized to study the limit profile of any general reversible Markov chain using its entire spectrum.

This paper focuses on determining the limit profile of a reversible Markov chain by comparing to another Markov chain on the same configuration space whose limit profile is known. The goal of this paper is to prove that the star transposition shuffle exhibits the same limit profile behavior as random transpositions. The assumptions needed for this new theory are that the two Markov chains are simultaneously diagonalizable, and that they both exhibit $\ell_2$ and total variation distance cutoff. The technique introduced in the present paper differs from the traditional comparison theory developed by Diaconis and Saloff-Coste (\cite{DS}, \cite{DS2}); it  compares total variation distances via purely algebraic techniques, while the Diaconis and Saloff-Coste technique compares $\ell^2$ norms via path techniques. 

We now give a general introduction on mixing times, cutoff and limit profiles. Let $X$ be a finite state space with $\vert X \vert=n$, and let $P$ be the transition matrix of an aperiodic and irreducible Markov chain. In other words, the entry $P^t(x,y)$ is the probability of the walk starting at $x$ and being at $y$ after $t$ steps, for every $t \in \mathbb{N}$. The measure $P_x^t(\cdot)=P^t(x, \cdot)$ converges to a unique measure $\pi(\cdot)$ on $X$ as $t$ goes to infinity.
We study this convergence with respect to total variation distance, which is defined as
$$d_x(t)=\Vert P_x^t- \pi \Vert_{\textup{T.V.}} := \frac{1}{2} \sum_{y \in X} \vert P_x^t(y)- \pi(y) \vert .$$
 We set 
\[d(t)=  \max_{x \in X} \lbrace d_x(t) \rbrace .\]
\begin{definition}
The mixing time with respect to the total variation distance is defined as 
$$t_{\textup{mix}}(\varepsilon):= \min \{ t :   d(t)   \leq \varepsilon \},$$
for every $\varepsilon \in (0,1)$.
\end{definition}

Cutoff describes a phase transition: as we run the family of Markov chains, the total variation distance is almost equal to $1$, and then suddenly it drops and approaches zero as $n$ grows. We now give the formal definition of cutoff.
\begin{definition}
A family of Markov chains is said to have cutoff at time $t_n$ with window $w_n=o(t_n)$ if and only if
$$\lim_{c \rightarrow \infty} \lim_{n \rightarrow \infty} d^{(n)}(t_n-cw_n)= 1 \mbox{ and } \lim_{c \rightarrow \infty} \lim_{n \rightarrow \infty} d^{(n)}(t_n+cw_n)= 0,$$ 
where $d^{(n)}(t)$ denotes the total variation distance of the $n$--th Markov chain.
\end{definition}

Given a Markov chain exhibiting cutoff, one can ask for more precise control on the exact distance from stationarity. This is known as the limit profile, defined as:
\[ \Phi_x(c):=\lim_{n \rightarrow \infty} d_x^{(n)}\left( t_{n} + c w_{n} \right), \mbox{ for all $c \in \mathbb{R}$. } \]
If this limit does not exist, similar definitions apply for the $\limsup$ and the $\liminf$.

The limit profile is known for only a few  Markov chains, such as  the riffle shuffle \cite{BaD}, the asymmetric exclusion process on the segment \cite{BuN}, the simple exclusion process on the cycle \cite{Lacoin}, and the simple random walk on Ramanujan graphs \cite{LP}, etc. Teyssier \cite{Teyssier} determined the limit profile for random transpositions. Using representation theory of the symmetric group $S_n$, he used Fourier transform arguments for studying limit profiles that work for random walks on groups using a generating set that is a conjugacy class. In particular, he proved that for random transpositions, if $t= \frac{1}{2}n \log n + cn$, then
\[\Phi_x(c)=\Vert \textup{Poiss}(1+e^{-2c})- \textup{Poiss}(1) \Vert_{\textup{T.V.}},\]
for every $x\in S_n$ and $c \in \mathbb{R}$. In \cite{NT}, this limit profile behavior is proven to hold for the $k$-cycle card shuffle, under the assumption that $k=o(n/ \log n)$.

Our main result studies the limit profile of the star transpositions shuffle, which is not generated by a conjugacy class. One step of the star transpositions shuffle consists of picking a card of the deck uniformly at random and transposing it with the top card.  Flatto, Odlyzko and Wales \cite{FOW} found the eigenvalues of the transition matrix using representation theory of the symmetric group. Diaconis analysed the eigenvalue behavior in \cite{SFlour} to prove that star transpositions exhibit cutoff at $n \log n $ with window of order $n$. The following theorem discusses the limit profile of star transpositions.

\begin{theorem}\label{star}
For the star transpositions card shuffle at time $t= n(\log n +c)$, we have that
\[ \Phi_x(c) =d_{\textup{T.V.}} (\textup{Poiss}(1+e^{-c}), \textup{Poiss}(1)), \]
for every $x\in S_n$ and $c \in \mathbb{R}$.
\end{theorem}
The proof of Theorem \ref{star} requires a new idea, since Teyssier's technique works for conjugacy invariant random walks and the more general technique analyzed in \cite{NT} assumes the knowledge of certain eigenfunctions. The main ingredient of the proof of Theorem \ref{star} is that star transpositions and random transpositions are simultaneously diagonalizable. This allows us to only use the eigenvalues, found in \cite{FOW}, and not the eigenfunctions of star transpositions. We also need to use the eigenvalues \cite{DiaSha} and the limit profile \cite{Teyssier} of random transpositions.

To prove Theorem \ref{star}, we develop a technique that allows us to compare the limit profile of the shuffle in question to the limit profile of the random transpositions shuffle.
The following main lemma of this paper allows us to compare limit profiles of reversible Markov chains on the same space $X$, with the same stationary measure. 
\begin{lemma} \label{l2}
Let $P$ and $Q$ be the transition matrices of two reversible Markov chains on a finite space $X$ that share the same eigenbasis and stationary measure $\pi$. In particular, we denote by $f_i: X \rightarrow \mathbb{C}$ the orthonormal eigenvectors satisfying
\[Pf_i= \beta_i f_i \mbox{ and } Qf_i= q_i f_i, \]
where $i=1, \ldots, \vert X \vert$ and $\beta_1=q_1=1$. We have that 
\[4 \Vert P^t_x-Q^{t_*}_x \Vert^2_{T.V.} \leq \sum_{i=2}^{\vert X \vert} f_i(x)^2 (\beta_i^{t}- q_i^{{t_*}})^2,\]
for every $x\in X$. If $P,Q$ are the transition matrices of transitive Markov chains then
\[4 \Vert P^t_x-Q^{t_*}_x \Vert^2_{T.V.} \leq \sum_{i=2}^{\vert X \vert} (\beta_i^{t}- q_i^{{t_*}})^2,\]
for every $x\in X$.
\end{lemma}
Letting $ n$ approach infinity, we obtain the following general result.
\begin{theorem}\label{main}
Let $P$ and $Q$ be the transition matrices of two reversible Markov chains on a finite space $X$ that satisfy the original assumptions of Lemma \ref{l2}. Assume $P$ exhibits cutoff at $t_n$ with window $w_n$ and $Q$ exhibits cutoff at $\overline{t}_n$ with window $\overline{w}_n$.  For $t = t_n + c w_n$ and $t_* = \overline{t}_n + c \overline{w}_n$, we have that  
\begin{equation}\label{ei}
 \vert \overline{\Phi}_x(c) - \Phi_x(c)  \vert \leq \frac{1}{2}  \lim_{n \rightarrow \infty}\left(\sum_{i=2}^{\vert X \vert} f_i(x)^2 (\beta_i^{t}- q_i^{{t_*}})^2 \right)^{1/2},
\end{equation}
for every $c \in \mathbb{R}$, $ x \in X$.
If $P,Q$ are furthermore the transition matrices of transitive Markov chains then
\begin{equation}\label{ei2}
 \vert \overline{\Phi}_x(c) -\Phi_x(c)  \vert \leq \frac{1}{2}  \lim_{n \rightarrow \infty}\left(  \sum_{i=2}^{\vert X \vert} (\beta_i^{t}- q_i^{{t_*}})^2 \right)^{1/2} ,
\end{equation}
for every $ c \in \mathbb{R}$, $ x\in X$. 
\end{theorem}

The assumptions of Lemma \ref{l2} are equivalent to saying that $P$ and $Q$ commute, i.e. $PQ=QP$.
Even though the assumptions of Lemma \ref{l2} and Theorem \ref{main} sound restrictive, it is actually not rare that two reversible random walks on the same group are simultaneously diagonalizable.  As explained in Section \ref{rtra}, it is a standard fact that random transpositions commutes with any other symmetric random walk on $S_n$. In general, the transition matrix of a symmetric, conjugacy invariant random walk on a group $G$ commutes with any other symmetric transition matrix of a random walk on $G$ (Laurent Saloff-Coste, personnal communication).  

In Section \ref{op}, we discuss possible conjectures and the difficulties that occur when trying to apply Theorem \ref{main} for the case of random-to-random.
Lemma \ref{l2} and Theorem \ref{main} are proven in Section \ref{next}. The proof of Theorem \ref{star} is contained in Section \ref{st}. Section \ref{rtra} summarizes the main tools that we need from representation theory and random transpositions.

\section{The proof of the main lemma}\label{next}
\begin{proof}[Proof of Lemma \ref{l2}]
Since $P$ and $Q$ are both reversible Markov chains with respect to the same stationary measure $\pi$ and share the same orthonormal eigenbasis $f_i$, we can write
\[\frac{P^t_x(y)}{\pi(y)} = 1 + \sum_{j=2}^{\vert X \vert} f_j(x) f_j(y) \beta_j^t\]
and similarly 
\[\frac{Q^t_x(y)}{\pi(y)} = 1 + \sum_{j=2}^{\vert X \vert} f_j(x) f_j(y) q_j^{t},\]
as explained in Lemma 12.2 of \cite{LPW}. The orthonormality of the $\lbrace f_i \rbrace$ gives that
\begin{equation}\label{cs}
d_{2,x}(t)=\bigg \Vert \frac{Q^t_x-P^t_x}{\pi} \bigg \Vert_2^2= \sum_{j=2}^{\vert X \vert} f_j(x)^2 (\beta_j^t- q_j^{t_*})^2,
\end{equation}
where the $\ell^2$ norm is defined as
\[\Vert g \Vert_2^2= \sum_{x \in X} g^2(x) \pi(x).\]
Equation \eqref{cs} and Cauchy-Schwartz give \eqref{ei}.

Note that if $P$ and $Q$ are transitive, then $d_x(t)=d(t)$ for every $x \in X$. Summing over $x\in X$ on both sides of \eqref{ei}, we get
\[n  \vert \overline{\Phi}_x(c) - \Phi_x(c)  \vert \leq  \sum_{x \in X}\sum_{i=2}^{\vert X \vert} f_i(x)^2(\beta_j^t- q_j^{t_*})^2\]
The standard fact that $\sum_{x \in X}f_i(x)^2=n$ finishes the proof of \eqref{ei2}.
\end{proof}

Theorem \ref{main} follows from Lemma \ref{l2} by taking limits on both sides.

\section{Representation theory of the symmetric group and Random transpositions}\label{rtra}
In this section, we summarize features of the random transposition shuffle that are important for the analysis of the limit profile of star transpositions. The transition matrix of random transpositions is given by
\[Q(x,y)= \begin{cases}
\frac{1}{n} & \mbox{ if } y=x,\\
\frac{2}{n^2} & \mbox{ if } y=xs, \mbox{ where $s$ is a transposition,} \\
 0 & \mbox{ otherwise.}
\end{cases}\]
It is easy to check that $Q$ is a symmetric matrix. Diaconis and Shahshahani \cite{DiaSha} found all the eigenvalues of $Q$ and their multiplicities. Their formulas are in terms of standard Young tableaux, which we now introduce. 

By a \emph{partition} $\lambda=(\lambda_1,\ldots, \lambda_k)$ of $n$ we mean a sequence of natural numbers such that $\lambda_1\geq \lambda_2 \geq\ldots \geq \lambda_k>0$ and $\sum_{i=1}^k \lambda_i=n$. Every partition coressponds to a \emph{Young diagram}, which has $\lambda_i$ boxes in row $i$. We will also consider the transpose partition $\lambda'$ of $\lambda$ to be the partition occuring by transposing the Young diagram of $\lambda$.
\begin{example}
We can think of $\lambda=(3,2)$ as
$\lambda=
\young(~~~,~~)
$.
In this case, $\lambda'= (2,2,1)=
\young(~~,~~,~)
$.
\end{example}
To describe the eigenvalues of many shuffles and their multiplicities, we will need the notion of \emph{standard Young tableaux} (SYT) of type $\lambda$, which is a filling of the Young diagram $\lambda$ with all the numbers in $[ n]$, so that the entries of each row and column of the diagram appear in increasing order.
\begin{example}
One SYT of $\lambda=(3,2)$ is
$\lambda=
\young(124,35)
$.
\end{example}

The following notion will be important when counting the multiplicities of the eigenvalues of $Q$.
\begin{definition}
Let $\lambda$ be a \emph{partition} of $n$. We define $d_{\lambda}$ to be the number of SYT of $\lambda$.
\end{definition}

The following lemma provides a useful bound on $d_{\lambda}.$
\begin{lemma}[Corollary 2 of \cite{DiaSha}] \label{j2}
Let $\lambda $ be a partition of $n$. Let $j=n- \lambda_1$, then
$d_{\lambda} \leq {n \choose j} \sqrt{j!}$.
\end{lemma}

The following lemma presents the eigenvalues of random transpositions.
\begin{lemma}[Lemma 7 of \cite{DiaSha}]\label{rteig}
Let $\lambda=(\lambda_1,\ldots, \lambda_k)$ be a partition of $n$. The eigenalues of $Q$ are given by
\[s_{\lambda}= \frac{1}{n} + \frac{n-1}{n} r_{\lambda},\]
where 
\[r_{\lambda} = \frac{1}{{n \choose 2}} \sum_{i=1}^k\bigg [{\lambda_i \choose 2} - {\lambda_i' \choose 2} \bigg].\]
Each eigenvalue $s_{\lambda}$ occurs with multiplicity $d_{\lambda}^2$.
\end{lemma}
The formula for the $s_{\lambda}$ gives the following standard fact which is used in both \cite{DiaSha} and \cite{Teyssier}.
\begin{corollary}\label{j}
Let $\lambda $ be a partition of $n$. If $j=n- \lambda_1$ is constant, then
\[s_{\lambda}= 1-\frac{2j}{n} +O \left(\frac{1}{n^2}\right). \]
Similarly, if $j=n- \lambda'_1$ is constant, then
\[s_{\lambda}= -1+\frac{2(j+1)}{n} +O \left(\frac{1}{n^2}\right). \]
\end{corollary}

Finally, the following lemma discusses the limit profile of random transpositions.
\begin{lemma}[Teyssier \cite{Teyssier}]
For random transpositions, we have that at $t= \frac{1}{2} n (\log n +c)$
\[\overline{\Phi}(c)=d_{T.V.} (\textup{Poiss}(1+e^{-c}), \textup{Poiss}(1)),\]
where $c \in \mathbb{R}$.
\end{lemma}

The following lemma will allow us to compare the limit profile of star transpositions to the limit profile of random transpositions. As pointed out by Laurent Saloff-Coste, it is a standard fact that the transition matrix of a symmetric random walk on a group $G$ generated by a union of conjugacy classes commutes with the transition matrix of any symmetric random walk on $G$. For completeness, we include the proof of this statement for the case of $Q$.
\begin{lemma}\label{commute}
Let $P$ be the transition matrix of a symmetric random walk on $S_n$. The transition matrix $Q$ of random transpositions commutes with $P$.
\end{lemma}
\begin{proof}
Let $\mu$ be the probability measure on $S_n$, such that $P(x,w)= \mu(x^{-1} w)$ and define $\nu$ to be the probability measure on $S_n$, such that $Q(x,w)= \nu(x^{-1} w)$.

Let $x, y \in S_n$. We write
\[(PQ)(x,y)=\sum_{w \in S_n}P(x,w) Q(w,y) .\]
Let $S'$ be the set of all transpositions and let $S=S' \cup \{\text{id}\}$. We have that
\[(PQ)(x,y)=\sum_{s \in S}P(x,ys) Q(ys,y)=\sum_{s \in S}\mu(x^{-1} ys)\nu(s^{-1}). \]
We want to use the fact that $\nu $ is constant on all transpositions, which form a conjugacy class. Because of this, we write
\[(PQ)(x,y)=\sum_{s \in S}\mu((x^{-1} y)s(y^{-1}x )(x^{-1}y))\nu(s^{-1}). \]
Let $\overline{s}= (x^{-1} y)s(y^{-1}x  )\in S$ and recall that $\nu(s)=\nu(s^{-1})= \nu(\overline{s})= \nu(\overline{s}^{-1})$.
Therefore,
\[(PQ)(x,y)=\sum_{\overline{s} \in S}\mu(\overline{s}x^{-1}y)\nu(\overline{s}^{-1})=\sum_{\overline{s} \in S} P(x \overline{s}^{-1}, y)  Q(x , x \overline{s}^{-1})= (QP)(x,y), \]
which finishes the proof. 
\end{proof}

\section{Star transpositions}\label{st}
The goal of this section is to prove Theorem \ref{star} using Theorem \ref{main}. 

In the following lemma, we recall the eigenvalues of star transpositions.
\begin{lemma}[Theorem 3.7 of \cite{FOW}]\label{steig}
Let $\lambda =(\lambda_1,\ldots, \lambda_k)$ be a partition of $n$. Let $\lambda^{(i)}$ be a partition of $n-1$ occurring by deleting the last box in the $i$-th row of $\lambda$ (if this results to a partition). The eigenvalues of $P$ corresponding to $\lambda$ are given by $\overline{s}_{\lambda^{(i)}}=\frac{1}{n}(\lambda_i -i +1) $, and they occur with multiplicity $d_{\lambda}d_{\lambda^{(i)}}$.
\end{lemma}

To prove Theorem \ref{star} we will need the following lemma.
\begin{lemma}\label{dim}
Let $\lambda$ be a partition of $n$. Let $i>1$ and denote $j= n- \lambda_1$. Then 
\[d_{\lambda^{(i)}} \leq \frac{4^j}{n} d_{\lambda} \quad \mbox{ and } \quad - \frac{j}{n} \leq \overline{s}_{\lambda^{(i)}}  \leq \frac{n-j}{n}.\]
\end{lemma}
\begin{proof}
In the diagram associated to $\lambda$, the hook $h_{i,j}$ of the $(i,j)$ box is the number of boxes which are below or on the right of
our box (including our box). We call $\equ( \lambda)$ the product of the hooks of the partition $\lambda$. It is a standard fact that
 \[d_{\lambda}=\frac{n!}{\equ( \lambda)}.\]
For example, the entry of each box in
$\lambda=
\young(431,21)
$ gives the corresponding hook. Then $d_{\lambda}=5$.

For $\lambda^{(i)}$ to be a valid partition, we have that $h_{i, \lambda_i}=1$.  We have that
\[\frac{d_{\lambda^{(i)}}}{d_{\lambda}} = \frac{1}{n} \prod_{\ell=1}^{i-1} \frac{h_{\ell,\lambda_{i}} }{h_{\ell,\lambda_{i}} -1}  \prod_{k=1}^{\lambda_i-1} \frac{ h_{i,k} }{h_{i,k}-1 }.\]
Notice that each $h_{a,b}$ appearing in the above products has to be at least equal to $2$, since the $(i, \lambda_i)$ box and the $(a,b)$ box contribute to $h_{a,b}$. Therefore, each ratio $ \frac{ h_{a,b} }{h_{a,b}-1 }$ is at most $2$, which gives that
\[d_{\lambda^{(i)}}\leq \frac{2^{i + \lambda_i-2}}{n} d_{\lambda}. \]
Using the fact that $i+ \lambda_i \leq 2j+2$ for every $i>1$ finishes the proof of the first statement.

The final claim holds because $i \leq j+1$ and therefore
$$ -\frac{j}{n} \leq \overline{s}_{\lambda^{(i)}}=\frac{1}{n}(\lambda_i -i +1)  \leq  \frac{\lambda_1}{n} =  \frac{n-j}{n},$$
for every $i>1$.
\end{proof}

The following standard fact is a consequence of the branching rule for the irreducible representations of $S_n$. We refer to Theorem 3.6 of \cite{FOW} for the exact statement of the branching rule.
\begin{lemma}\label{branch}
Let $\lambda $ be a partition of $n$. Then
$d_{\lambda}= \sum_i d_{\lambda^{i}}$ and $d_{\lambda^{(i)}} = d_{\lambda'^{(\lambda_i)}} $.
\end{lemma}

We are now ready to prove Theorem \ref{star}.
\begin{proof}[Proof of Theorem \ref{star}]
Lemma \ref{commute} allows us to use Theorem \ref{main}.
Therefore, the goal is to prove that 
\[
  \lim_{n \rightarrow \infty} \sum_{\lambda } d_{\lambda} \sum_{i} d_{\lambda^{(i)}}(s_{\lambda}^{t}- \overline{s}_{\lambda^{(i)}}^{{t_*}})^2 =0,\] 
where $s_{\lambda}$ and $\overline{s}_{\lambda^{(i)}}$ are the eigenvalues of random transpositions and star transpositions respectively. Recalling Lemmas \ref{rteig} and \ref{steig}, we have 
\[s_{\lambda} = \frac{1}{n} + \frac{n-1}{n} r_{\lambda} \text{ and } \overline{s}_{\lambda^{(i)}}=\frac{1}{n} + \frac{n-1}{n} \overline r_{\lambda^{(i)}},\]
where 
\[r_{\lambda} = \frac{1}{{n \choose 2}} \sum_{i=1}^k\bigg [{\lambda_i \choose 2} - {\lambda_i' \choose 2} \bigg] \text{ and } \overline{r}_{\lambda^{(i)}}= \frac{\lambda_i-i}{n-1}.\]
Note that 
\begin{equation}\label{transpose}
r_{\lambda'}= - r_{\lambda}  \mbox{ and }\overline r_{(\lambda')^{(\lambda_i)}}= - \overline{r}_{\lambda^{(i)}},
\end{equation}
where $\lambda'$ is the transpose diagram of $\lambda$.

Our strategy is to prove that for every $\varepsilon>0$ and $c \in \mathbb{R}$, there is an $M=M(c, \varepsilon)$ such that
\begin{enumerate}
\item $\sum_{\lambda_1 \leq n-M} d_{\lambda}^2 \vert s_{\lambda} \vert^{2t} \leq \varepsilon,  $
\item $\sum_{\substack{\lambda_1 \leq n-M \\ \lambda'_1 \leq n-M }} d_{\lambda} \sum_{i \geq 1} d_{\lambda^{(i)}} \vert \overline{s}_{\lambda^{(i)}} \vert^{2t_*} \leq \varepsilon,  $
\item $\sum_{\substack{\lambda_1 \leq n-M \\ \lambda'_1 \leq n-M }} d_{\lambda}   \vert s_{\lambda} \vert^t  \sum_{i \geq 1} d_{\lambda^{(i)}} \vert \overline{s}_{\lambda^{(i)}} \vert^{t_*}\leq \varepsilon, \mbox{ and } $
\item $\sum_{\lambda_1 > n-M} d_{\lambda} \sum_{i} d_{\lambda^{(i)}} \vert s^t_{\lambda}- \overline{s}_{\lambda^{(i)}}^{t_*} \vert^{2} \leq \varepsilon , \mbox{ and }\sum_{\lambda'_1 > n-M} d_{\lambda} \sum_{i} d_{\lambda^{(i)}} \vert s^t_{\lambda}- \overline{s}_{\lambda^{(i)}}^{t_*} \vert^{2} \leq \varepsilon$
\end{enumerate}
for sufficiently large $n.$

Lemma 4.1 of Teyssier \cite{Teyssier} states that for every $\varepsilon \in (0,1)$ and $c \in \mathbb{R}$, there is an $M_1=M_1(c, \varepsilon)$ such that at $t= \frac{1}{2}n( \log n +c )$ we have that
\[\sum_{\lambda_1 \leq n-M_1} d_{\lambda} \vert s_{\lambda} \vert^t \leq \varepsilon,\]
where $n $ is sufficiently large.
This implies that
\[\sum_{\lambda_1 \leq n-M_1} d_{\lambda}^2 \vert s_{\lambda} \vert^{2t} \leq \varepsilon,\]
where $n $ is sufficiently large. This gives 1.

Now, we want to prove 2. Let $M_2=M_2(c, \varepsilon)$ be such that $ \sum_{j \geq M_2} \frac{e^{-2cj}}{j!}  \leq \varepsilon/2$. Equation \eqref{transpose} implies that if $\overline{r}_{\lambda^{(i)}}  \geq 0$ then $\overline{s}_{\lambda^{(i)}}= \vert \overline{s}_{\lambda^{(i)}} \vert \geq  \vert \overline{s}_{\lambda'^{(\lambda_i)}} \vert $. Therefore, it suffices to consider only the cases where $\overline{r}_{\lambda^{(i)}}  \geq 0$, which implies that $\overline{s}_{\lambda^{(i)}} \geq 0$. 
Lemmas \ref{j2}, \ref{dim} and \ref{branch} give 
\[\sum_{\substack{\lambda_1 \leq n-M_2 \\ \lambda'_1 \leq n-M_2 }} d_{\lambda} \sum_{\substack{i \geq 1  \\ \overline{s}_{\lambda^{(i)}} \geq 0 }} d_{\lambda^{(i)}} \vert \overline{s}_{\lambda^{(i)}} \vert^{2t_*} \leq  \sum_{j \geq M_2} d_{\lambda}^2 \left(1- \frac{j}{n} \right)^{2t} \leq    \sum_{j \geq M_2} \frac{n^{2j}}{j!} e^{-\frac{2tj}{n}} .\]
For $t=n \log n + cn$, we have 
\[\sum_{\substack{\lambda_1 \leq n-M_2 \\ \lambda'_1 \leq n-M_2 }} d_{\lambda}  \sum_{\substack{i \geq 1  \\ \overline{s}_{\lambda^{(i)}} \geq 0 }} d_{\lambda^{(i)}} \vert \overline{s}_{\lambda^{(i)}} \vert^{2t_*} \leq  \sum_{j \geq M_2} \frac{e^{-2cj}}{j!}  \leq \varepsilon/2.\]
This gives 2.

We now focus on 3. Let $M_3= \max{ \lbrace M_1, M_2 \rbrace }$. Just like in 2, it suffices to consider only the partitions $\lambda$ where $\overline{r}_{\lambda^{(i)}}  \geq 0$. In combination with Lemma \ref{dim}, this implies that $0 \leq \overline{s}_{\lambda^{(i)}}\leq  \overline{s}_{\lambda^{(1)}} = 1- \frac{j}{n}$, where $j=n- \lambda_1$. Therefore, Lemma \ref{branch} gives that
\[\sum_{\lambda_1 \leq n-M_3} d_{\lambda} \vert s_{\lambda} \vert^t  \sum_{\substack{i \geq 1  \\ \overline{s}_{\lambda^{(i)}} \geq 0 }} d_{\lambda^{(i)}} \vert \overline{s}_{\lambda^{(i)}} \vert^{t_*} \leq\sum_{\lambda_1 \leq n-M_3} d_{\lambda}^2 \vert s_{\lambda} \vert^t  \left( 1- \frac{j}{n} \right)^{t_*}.\]
Applying Cauchy-Schwartz on the right hand side gives 
\[(\sum_{\lambda_1 \leq n-M_3} d_{\lambda} \vert s_{\lambda} \vert^t  \sum_{\substack{i \geq 1  \\ \overline{s}_{\lambda^{(i)}} \geq 0 }} d_{\lambda^{(i)}} \vert \overline{s}_{\lambda^{(i)}} \vert^{t_*})^2 \leq \sum_{\lambda_1 \leq n-M_3} d_{\lambda}^2 \vert s_{\lambda} \vert^{2t}\sum_{\lambda_1 \leq n- M_3} d_{\lambda}^2 \left(1- \frac{j}{n} \right)^{2t_*} . \]

We set $M_3= \max{ \lbrace M_1, M_2 \rbrace }$ so that both $\sum_{\lambda_1 \leq n-M_3} d_{\lambda}^2 \vert s_{\lambda} \vert^{2t}$ and $\sum_{\lambda_1 \leq n-M_3} d_{\lambda}^2 \left(1- \frac{j}{n} \right)^{2t_*}$ are less than $\varepsilon$ as we saw in 1 and 2. This gives 3.

We now focus on proving 4. Let $M= \max{ \lbrace M_1,M_2,M_3\rbrace }$. We start by writing 
\[\sum_{\lambda_1 > n-M} d_{\lambda} \sum_{i} d_{\lambda^{(i)}} \vert s^t_{\lambda}- \overline{s}_{\lambda^{(i)}}^{t_*} \vert^{2} = \sum_{\lambda_1 > n-M} d_{\lambda} \sum_{i>1} d_{\lambda^{(i)}} \vert s^t_{\lambda}- \overline{s}_{\lambda^{(i)}}^{t_*}  \vert^2+ \sum_{\lambda_1 > n-M} d_{\lambda} d_{\lambda^{(1)}} \vert s^t_{\lambda}- \overline{s}_{\lambda^{(1)}}^{t_*} \vert^{2} . \]
 
Notice that for  $j=n-\lambda_1$, with $j$ constant Lemma \ref{dim} gives that $\vert \overline{s}_{\lambda^{(i)}} \vert \leq \max \{ \frac{j}{n} , 1-\frac{j}{n}\}=1-\frac{j}{n} $, for sufficiently large $n$. Therefore,
\begin{align} \label{per}
\sum_{\lambda_1 > n-M} d_{\lambda} \sum_{i>1} d_{\lambda^{(i)}} \vert s^t_{\lambda}- \overline{s}_{\lambda^{(i)}}^{t_*} \vert^{2} & \leq \frac{4^MM}{n} \sum_{j<M} d_{\lambda}^2 \left( \vert s^{t}_{\lambda} \vert+ \left(1- \frac{j}{n} \right)^{t_*} \right)^2
\end{align}
for sufficient large $n$. Using Corollary \ref{j},  and the facts that $t= \frac{1}{2}n(\log n +c)$ and $t_*= n(\log n +c)$, we have that
\[ \left( 1- \frac{j}{n} \right)^{t_*}  \leq\frac{e^{-cj}}{n^j} \mbox{ and } s^{t}_{\lambda}= \left(1-\frac{2j}{n} +O \left(\frac{1}{n^2}\right) \right)^t \leq 3\frac{e^{-cj}}{n^j}  \]
for sufficiently lareg $n$. Therefore, we have that
\begin{equation}\label{ff}
\eqref{per} 
 \leq \frac{4^{M+2}M}{n} \sum_{j<M} \frac{e^{2cj}}{j!}  \leq \varepsilon/2,
\end{equation}
for sufficient large $n$.

Finally, Corollary \ref{j} and Lemma \ref{branch} give that
\begin{align*}
\sum_{\lambda_1 > n-M} d_{\lambda} d_{\lambda^{(1)}} \vert s^t_{\lambda}- \overline{s}_{\lambda^{(1)}}^{t_*} \vert^{2} & \leq \sum_{j<M} d_{\lambda}^2 \left(  \left(1-\frac{2j}{n} +O \left(\frac{1}{n^2}\right) \right)^t - \left(1-\frac{j}{n} \right)^{t_*} \right)^2\cr
&  \leq \sum_{j<M} d_{\lambda}^2 \left(  \frac{e^{-cj}}{n^j}\left(1+ O \left( \frac{\log n}{n} \right) \right)- \frac{e^{-cj}}{n^j}\left(1+O\left(\frac{j^2}{n} \right) \right) \right)^2,
\end{align*}
where in the last inequality, we plug in the values for $t$ and $t_*$ and we use the fact that $ e^{x} \left(1- \frac{x^2}{n} \right)\leq \left( 1+\frac{x}{n}\right)^n \leq e^{x}$ for $\vert x \vert \leq n$. Using Lemma \ref{j2}, we get that
\begin{align*}
\sum_{\lambda_1 > n-M} d_{\lambda} d_{\lambda^{(1)}} \vert s^t_{\lambda}- \overline{s}_{\lambda^{(1)}}^{t_*} \vert^{2} =O \left(\frac{\log^2 n}{n^2}\right).
\end{align*}
Therefore 
\begin{align}\label{fff}
\sum_{\lambda_1 > n-M} d_{\lambda} d_{\lambda^{(1)}} \vert s^t_{\lambda}- \overline{s}_{\lambda^{(1)}}^{t_*} \vert^{2} \leq \varepsilon/2,
\end{align}
for sufficiently large $n$.
Equations \eqref{ff} and \eqref{fff} give the first inequality in 4. To prove the second inequality, we write
\[\sum_{\lambda'_1 > n-M} d_{\lambda} \sum_{i} d_{\lambda^{(i)}} \vert s^t_{\lambda}- \overline{s}_{\lambda^{(i)}}^{t_*} \vert^{2} = \sum_{\lambda'_1 > n-M} d_{\lambda'} \sum_{i} d_{\lambda'^{(\lambda_i)}} \bigg \vert s^t_{\lambda}- \left( \frac{2}{n}-\overline{s}_{\lambda'^{(\lambda_i)}}\right)^{t_*}\bigg \vert^{2} \]
and we proceed similarly as before, using the second part of Corollary \ref{rteig} and the fact that $\overline{s}_{\lambda'^{(1)}}= \frac{\lambda'_1}{n}$. We note that at this point it is important that $t$ and $t_*$ are either both odd or both even, so that the corresponding eigenvalues have the same sign. To achieve this, we could have considered $t$ to be $\frac{1}{2}n (\log n+c) +1 $, which doesn't affect the limit profile behavior.
\end{proof}

\section{Open Questions}\label{op}
\subsection{Conjugacy class shuffles with cutoff}
We are now looking at the case where $G=S_n$ and we want to apply Theorem \ref{main} when $Q$ is the transition matrix of random transpositions. 
Let $P$ be the transition matrix of a lazy random walk on $S_n$ generated by a conjugacy class (or a union of conjugacy classes). 

A standard application of Schur's lemma gives that
\[D=
\begin{pmatrix}
 \hat{P}(\rho_1)
  & \rvline & \bigzero  & \rvline & \bigzero \\
\hline
 \bigzero & \rvline & \ldots &  \rvline &  \bigzero\\
\hline
  \bigzero  & \rvline & \bigzero & \rvline & 
  \hat{P}(\rho_k)
\end{pmatrix},
\]
is a diagonal matrix and in fact each Fourier transform $\hat{P}(\rho_i)$ is a multiple of the identity. 
Since each irreducible representation contributes only one eigenvalue, we form the following conjecture.
\begin{conjecture}
Let $P$ be the transition matrix of a lazy random walk on $S_n$ generated by a conjugacy class (or a union of conjugacy classes) with number of fixed points of order $n-o(n)$. If $P$ exhibits total variation cutoff and $\ell^2$ cutoff at $t_n$ with window $w_n$, then 
\[ \Phi_x(c) = d_{T.V.} (\mbox{Poiss}(1+e^{-c}), \mbox{Poiss}(1)).\]
\end{conjecture}
To support this claim, we point out that the card shuffle generated by $k$--cycles with $k=o(n)$ satisfies this conjecture as proven in \cite{NT}.
\subsection{Random-to-Random}
One step of the random-to-random card shuffle consists of picking a card and a position of the deck uniformly and independently at random and moving that card to that position. 
This shuffle was introduced by Diaconis and Saloff-Coste \cite{DS}, who proved that the mixing time is $O(n \log n)$. It has been studied by (\cite{JC}, \cite{LZ}, \cite{Lower}, \cite{MorrisQin}), since cutoff at $\frac{3}{4} n \log n - \frac{1}{4} n \log \log n $ was conjectured by Diaconis for fifteen years \cite{DiaConj}. Recently, Bernstein and Nestoridi \cite{BN} proved the upper bound for the mixing time, which in combination with Subag's \cite{Lower} lower bound resolved the desired conjecture.

The main open question of this section is the following.
\begin{op}
What is the limit profile of random-to-random?
\end{op}

Even though Theorem \ref{main} can be applied for random transpositions and random-to-random (due to Lemma \ref{commute}), it is not clear if it can lead to determining the limit profile of random-to-random. In fact the limit profile of random-to-random seems to be different from the limit profile of random transpositions. 

We can see this by looking at just at the $n-1$ dimensional representation $\rho$. For random transpositions, $\rho$ contribute the eigenvalue $1- \frac{2}{n} $ with multiplicity $n^2$. For random-to-random, we have $n^{3/2}$ eigenvalues that are roughly equal to $1- \frac{1}{n}$  and the rest are negligible. When we study the error, we get a term that is roughly equal to
\[n^{3/2}\left( \left(1- \frac{2}{n} \right)^t-  \left(1- \frac{1}{n} \right)^{t_*}\right)^2 + n^{2}\left(1- \frac{2}{n} \right)^{2t},\] 
where $t, t_ *$ are the corresponding cutoff times. This term does not converge to zero as $n$ approches infinity, thus suggesting that the limit profiles of the two shuffles are different.

\textbf{Acknowledgements.}
The author would like to thank Sam Olesker-Taylor, Sarah Peluse, Laurent Saloff-Coste and Dominik Schmid for helpful discussions and comments.

 \bibliographystyle{plain}
\bibliography{star}

\end{document}